\documentclass{amsart}

\usepackage[utf8]{inputenc}

\usepackage{amsmath, amsfonts,amssymb,amsthm,array}
\usepackage{graphicx}
\usepackage{stackengine}
\usepackage{thmtools,thm-restate}
\usepackage{enumitem}
\usepackage{float}
\usepackage{comment}
\usepackage{xifthen}
\usepackage{latexsym}
\usepackage{stmaryrd}
\usepackage{enumitem}
\usepackage[bookmarks,hyperfootnotes=false]{hyperref}
\usepackage{multirow}
\usepackage{xcolor}
\usepackage{caption}
\colorlet{darkishRed}{red!80!black}
\colorlet{darkishBlue}{blue!60!black}
\colorlet{darkishGreen}{green!60!black}
\usepackage{hyperref}
\hypersetup{
pdftitle = {Ends of Digraphs},
pdfsubject = {Ends of Digraphs},
pdfauthor = {Carl Bürger and Ruben Melcher},
pdfkeywords = {},
draft = false,
bookmarksopen=true}

\newtheorem{innercustomthm}{Theorem}

\newtheorem{innercustomlem}{Lemma}

\newenvironment{customthm}[1]
  {\renewcommand\theinnercustomthm{#1}\innercustomthm}
  {\endinnercustomthm}

\newenvironment{customlem}[1]
  {\renewcommand\theinnercustomlem{#1}\innercustomlem}
  {\endinnercustomlem}

\newtheorem{theorem}{Theorem}[section]
\newtheorem{mainresult}{Theorem}

\newtheorem{mainlemma}[mainresult]{Lemma}
\newtheorem{proposition}[theorem]{Proposition}

\newtheorem{lemma}[theorem]{Lemma}

\theoremstyle{definition}

\setenumerate{label={\normalfont (\roman*)},itemsep=0pt}

\colorlet{darkishGreen}{green!60!black}

\newcommand{\cB}{\mathcal{B}}
\newcommand{\cX}{\mathcal{X}}

\newcommand{\cU}{\mathcal{U}}

\newcommand{\cA}{\mathcal{A}}

\newcommand{\N}{\mathbb{N}}

\newcommand{\nao}[1][]{%
\ifthenelse{\equal{#1}{}}{\trianglelefteq_T}{\trianglelefteq_{T_{#1}}\!}%
}

\newcommand{\dc}[1]{\lceil #1\rceil}
\newcommand{\uc}[1]{\lfloor #1\rfloor}

\newcommand{\dS}{\mathstrut\mkern2.5mu S\mkern-13mu\raise1.4ex%
\hbox{$\leftrightarrow$}}

\def\lowfwd #1#2#3{{\mathop{\kern0pt #1}\limits^{\kern#2pt\raise.#3ex
\vbox to 0pt{\hbox{$\scriptscriptstyle\rightarrow$}\vss}}}}
\def\lowbkwd #1#2#3{{\mathop{\kern0pt #1}\limits^{\kern#2pt\raise.#3ex
\vbox to 0pt{\hbox{$\scriptscriptstyle\leftarrow$}\vss}}}}

\def\ve{\kern-1.5pt\lowfwd e{1.5}2\kern-1pt}
\def\ev{\kern-1pt\lowbkwd e{0.5}2\kern-1pt}
\def\vf{\kern-2pt\lowfwd f{2.5}2\kern-1pt}

\stackMath

{\begin{list}{}%
          {\setlength{\leftmargin}{#1}}%
          \item[]%
}
{\end{list}}

\graphicspath{{Main/}} 

\begin{document}

\title[Ends of digraphs]{Ends of digraphs I: basic theory} 
\author{Carl B\"{u}rger}
\author{Ruben Melcher}
\address{University of Hamburg, Department of Mathematics, Bundesstraße 55 (Geomatikum), 20146 Hamburg, Germany}
\email{carl.buerger@uni-hamburg.de, ruben.melcher@uni-hamburg.de}


\begin{abstract}
In a series of three papers we develop an end space theory for directed graphs. 
As for undirected graphs, the \emph{ends} of a digraph are points at infinity to which its rays converge. Unlike for undirected graphs, some ends are joined by  \emph{limit edges}; these are crucial for obtaining the end space of a digraph as a natural (inverse) limit of its finite contraction minors.

As our main result in this first paper of our series we show that the notion of \emph{directions} of an undirected graph, a tangle-like description of its ends, extends to digraphs: there is a one-to-one correspondence between the `directions' of a digraph and its ends and limit edges.

In the course of this we extend to digraphs a number of fundamental tools and techniques for the study of ends of graphs, such as the star-comb lemma and Schmidt's ranking of rayless graphs. \end{abstract}

\maketitle

\section{Introduction}

\subsection*{The series} 
Ends of graphs are one of the most important concepts in infinite graph theory. They can be thought of as points at infinity to which its rays converge. Formally, an \emph{end} of a graph $G$ is an equivalence class of its rays, where two rays are equivalent if for every finite vertex set $X\subseteq V(G)$ they have a tail in the same component of $G-X$. For example, infinite complete graphs or grids have one end, while the binary tree has continuum many ends, one for every rooted ray~\cite{DiestelBook5}. 
The concept of ends was introduced in 1931 by Freudenthal~\cite{Freudenthal1931Enden}, who defined ends for certain topological spaces. In 1964, Halin~\cite{Halin_Enden64} introduced ends for infinite undirected graphs, taking his cue directly from Carath\'eodory's \emph{Primenden} of regions in the complex plane~\cite{Caratheodory1913}. 

There is a natural topology on the set of ends of a graph $G$, which makes it into the \emph{end space $\Omega(G)$}. Polat~\cite{polat1996ends,polat1996ends2}  studied the topological properties of this space. Diestel and K\"uhn  \cite{CyclesI} extended this topological space to the space $\vert G \vert$ formed by the graph $G$ together with its ends. Many well known theorems of finite graph theory extend to this space $\vert G \vert$, while they do not generalise verbatim to infinite graphs. 
Examples include Nash-William's tree-packing theorem~\cite{DiestelBook3}, Fleischner's Hamiltonicity theorem \cite{AgelosFleischner}, and Whitney's planarity criterion~\cite{duality}. In the formulation of these theorems, topological arcs and circles take the role of paths and cycles, respectively.

For directed graphs, a similarly useful notion and theory of ends has never been found. There have been a few attempts, most notably by Zuther~\cite{Zuther96Thesis}, but not with very encouraging results. 
In this series we propose a new notion of ends of digraphs  and develop a corresponding theory of their end spaces. Let us give a brief overview of the series. 

In this first paper we lay the foundation for the whole series by extending to digraphs a number of techniques that are important in the study of ends of graphs. 

As our main result we show that the one-to-one correspondence between the directions and the ends of a graph has an analogue for digraphs. A \emph{direction} of a graph $G$ is a map $f$, with domain the set of finite vertex sets $X$ of $G$, that maps every such $X$ to a component of $G-X$ so that $f(X)\supseteq f(Y)$ whenever $X\subseteq Y$. Every end $\omega$ of $G$  naturally defines a direction $f_\omega$ which maps every finite vertex set $X \subseteq V(G)$ to the unique component of $G-X$ in which every ray representing $\omega$ has a tail. It is straightforward to show that $f_\omega$ is indeed a direction of $G$. Conversely, Diestel and K\"uhn~\cite{Diestel_Kuehn_Directions_03} proved that for every direction $f$ of $G$ there is a (unique) end $\omega$ of $G$ that defines $f$ in that $f_\omega=f$. This correspondence is now well known and has become a standard tool in the study of infinite graphs. See \cite{StarComb1StarsAndCombs,diestel2006end,EndsAndTangles,ApproximatingNormalTrees,EndsTanglesCrit,StoneCechTangles} for examples.

For a digraph $D$ we will adapt the definition of a direction by first replacing every occurrence of the word `component' with `strong component'. These directions of  $D$ will correspond bijectively to the ends of $D$. However, as there may be edges between distinct strong components of  $D$, there will be another type of direction: one that maps finite vertex sets $X\subseteq V(D)$ to the set of edges between two distinct strong components of $D-X$ in a compatible way. These latter directions of digraphs will correspond bijectively to its \emph{limit edges}---additional edges between distinct ends, or between ends and vertices, of a digraph. 

In the course of proving that the ends and limit edges of a digraph correspond to its two types of directions in this way, we extend to digraphs a number of fundamental tools and techniques for ends of graphs, such as the star-comb lemma \cite[Lemma~8.2.2]{DiestelBook5} and Schmidt's ranking of rayless graphs \cite{Schmidt1983}. 

In the second paper we will define a topology on the space $|D|$ formed by the digraph $D$ together with its ends and limit edges. To illustrate the typical use of this space $|D|$, we extend to it two statements about finite digraphs that do not generalise verbatim to infinite digraphs. The first statement is the characterisation of Eulerian digraphs by the condition that the in-degree of every vertex equals its out-degree. The second statement is the characterisation of strongly connected digraphs by the existence of a closed Hamilton walk, see \cite{bang2008digraphs}. In the course of our proofs we extend to the space $|D|$ a number of techniques that have become standard in proofs of statements about $\vert G \vert$, such as the jumping arc lemma or the fact that $|G|$ is an inverse limit of finite contraction minors of $G$.

In the third paper we consider normal spanning trees, one of the most important structural tools in infinite graph theory. Here a rooted tree $T\subseteq G$  is \emph{normal} in $G$ if the endvertices of every $T$-path in $G$ are comparable in the tree-order of $T$. (A~\emph{$T$-path} in $G$ is a non-trivial path that meets $T$ exactly in its endvertices.) In finite graphs, normal spanning trees are precisely the depth-first search trees \cite{DiestelBook5}.

As a directed analogue of normal spanning trees we introduce and study \emph{normal spanning arborescences} of digraphs. These are generalisations of depth-first search trees  to infinite digraphs, which promise to be as powerful for a structural analysis of digraphs as normal spanning trees are for graphs. We show that normal spanning arborescences capture the structure of the set of ends of the digraphs they span, both combinatorially and topologically.
Furthermore, we provide a Jung-type~\cite{jung69} criterion for the existence of normal spanning arborescences in digraphs.

\newpage
\subsection*{This paper.} In order to state the main results of this first paper of our series more formally, we need a few definitions.

A \emph{directed ray} is an infinite directed path that has a first vertex (but no last vertex). The directed subrays of a directed ray are its \emph{tails}. For the sake of readability we shall omit the word `directed' in  `directed path' and `directed ray' if there is no danger of confusion. We call a ray in a digraph $D$ \emph{solid} in $D$ if it has a tail in some strong component of $D-X$ for every finite vertex set $X\subseteq V(D)$. We call two solid rays in a digraph $D$ \emph{equivalent} if for every finite vertex set $X\subseteq V(D)$ they have a tail in the same strong component of $D-X$. The classes of this equivalence relation are the \emph{ends} of $D$. 
The set of ends of $D$ is denoted by $\Omega(D)$. In the second paper of this series we will equip $\Omega(D)$ with a topology and we will call $\Omega(D)$ together with this topology \emph{the end space of $D$}. Note that two solid rays $R$ and $R'$ in $D$ represent the same end if and only if $D$ contains infinitely many disjoint paths from $R$ to $R'$ and infinitely many disjoint paths from $R'$ to $R$.

For example, the digraph $D$ in Figure~\ref{fig:ladder} has two ends, which are shown as small dots on the right. Both the upper ray $R$ and the lower ray $R'$ are solid in $D$ because the vertex set of any tail of $R$ or $R'$ is strongly connected in $D$. Deleting finitely many vertices of $D$ always results in precisely two infinite strong components (and finitely many finite strong components) spanned by the vertex sets of tails of~$R$~or~$R'$.

\begin{figure}[h]
\centering
\def\svgwidth{0.8\columnwidth}
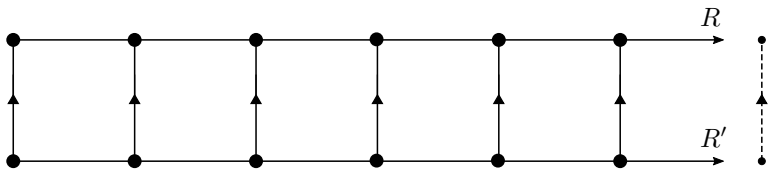
\caption{A digraph with two ends (depicted as small dots) linked by a limit edge (depicted as a dashed line). Every undirected edge in the figure represents a pair of inversely directed edges.}
\label{fig:ladder}
\end{figure}

Similarly to ends of graphs, the ends $\omega$ of a digraph can be thought of as points at infinity to which the rays that represent $\omega$ converge. We will make this formal in the second paper of our series, but roughly one can think of this as follows. For a finite vertex set $X\subseteq V(D)$ and an end $\omega\in\Omega(D)$ we write $C(X,\omega)$ for the unique strong component of $D-X$ that contains a tail of every ray that represents $\omega$; the end $\omega$ is then said to \emph{live} in that strong component.  In our topological space the strong components of the form $C(X,\omega)$ together with all the ends that live in them will essentially form the basic open neighbourhoods around $\omega$.

Given an infinite vertex set $U\subseteq V(D)$, we say that an end $\omega$ is \emph{in the closure} of $U$ in $D$ if $C(X,\omega)$ meets $U$ for every finite vertex set $X\subseteq V(D)$. (It will turn out that an end is in the closure of $U$ in $D$ if and only if it is in the topological closure~of~$U$.)

For undirected graphs $G$ one often needs to know whether an end $\omega$ is in the closure of a given vertex set $U$, i.e., whether $U$ meets $C(X,\omega)$ for every finite vertex set $X\subseteq V(G)$. This is equivalent to $G$ containing a comb with all its teeth in $U$. Recall that a \emph{comb} is the union of a ray $R$ (the comb's \emph{spine}) with infinitely many disjoint finite paths, possibly trivial, that have precisely their first vertex on~$R$. The last vertices of those paths are the \emph{teeth} of this comb. A standard tool in this context is the star-comb lemma~\cite[Lemma~8.2.2]{DiestelBook5} which states that a connected graph contains for a given set $U$ of vertices either a comb with all its teeth in $U$ or an infinite subdivided star with all its leaves in $U$. In this paper we will prove a directed version of the star-comb lemma.

Call two statements $A$ and $B$ \emph{complementary} if the negation of $A$ is equivalent to~$B$.  For a graph $G$, the statement that $G$ has an end in the closure of $U\subseteq V(G)$ is complementary to the statement that $G$ has a $U$-rank, see~\cite{StarComb1StarsAndCombs}.
For $U=V(G)$,  the $U$-rank is known as Schmidt's ranking of rayless graphs~\cite{DiestelBook5, Schmidt1983}. It is a standard technique to prove statements about rayless graphs by transfinite induction on Schmidt's rank. For example Bruhn, Diestel, Georgakopoulos, and Spr\"{u}ssel \cite{UnfriendlyPartition} employed this technique to prove the unfriendly partition conjecture for countable rayless graphs.

The directed analogue of a comb with all its teeth in $U$ will be a `necklace' attached to $U$. The \emph{symmetric ray} is the digraph obtained from an undirected ray by replacing each of its edges by its two orientations as separate directed edges. A \emph{necklace} is an inflated symmetric ray with finite branch sets. (An inflated $H$ is obtained from a digraph $H$ by subdividing some edges of $H$ finitely often and then replacing the `old' vertices by strongly connected digraphs. The \emph{branch sets} of the inflated $H$ are these  strongly connected digraphs.  See Section~\ref{section: preliminaries} for the formal definition of inflated, and of branch sets.) Figure~\ref{fig: necklace} shows an example of a necklace. 
\begin{figure}[h]
\centering 
\def\svgwidth{0.8\columnwidth}
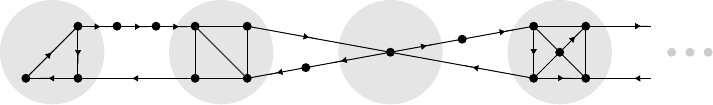
\caption{A necklace up to the fourth branch set. Every undirected edge in the figure represents a pair of inversely directed edges.}

\label{fig: necklace}
\end{figure} 
Given a set $U$ of vertices in a digraph $D$, a necklace $N\subseteq D$ is \emph{attached to} $U$ if infinitely many of the branch sets of $N$ contain a vertex  from $U$. We will see that the statement that  $D$ has an end in the closure of $U$ is equivalent to the statement that $D$ contains a necklace attached to $U$ as a subdigraph.

We extend Schmidt's result that a graph is rayless if and only if it has a rank. See Section~\ref{section: ends} for the definition of `$U$-rank' in digraphs.  
\begin{customlem}{\ref{lemma: necklace lemma}}[Necklace Lemma]
Let $D$ be any digraph and $U$ any set of vertices in~$D$. Then the following statements are complementary: \begin{enumerate}
    \item $D$ has a necklace attached to $U$\!; 
    \item $D$ has a $U$-rank. 
\end{enumerate}
\end{customlem}

Let us now define a directed analogue of the directions of undirected infinite graphs. Consider any digraph $D$, and write $\cX(D)$ for the set of finite vertex sets in $D$. A \emph{\emph{(}vertex-\emph{)}direction} of $D$ is a map $f$ with domain $\cX(D)$ that sends every $X\in\cX(D)$ to a strong component of $D-X$ so that $f(X)\supseteq f(Y)$ whenever $X\subseteq Y$. Ends of digraphs define vertex-directions in the same way as ends of graphs do; for every end $\omega\in \Omega(D)$ we write $f_\omega$ for the vertex-direction that maps every $X\in \cX(D)$ to the strong component $C(X,\omega)$ of $D-X$. We will show that this correspondence between ends and vertex-directions is bijective:

\begin{customthm}{\ref{thm: directions}} Let $D$ be any infinite digraph. The map $\omega \mapsto f_\omega$ with domain $\Omega(D)$ is a bijection between the ends and the vertex-directions of $D$. 
\end{customthm}
While most of the concepts that we investigate have undirected counterparts, there is one important exception: limit edges. If $\omega$ and $\eta$ are distinct ends of a digraph, there exists a finite vertex set $X\in \cX(D)$ such that $\omega$ and $\eta$ live in distinct strong components of $D-X$. Let us say that such a vertex set $X$ \emph{separates} $\omega$ and~$\eta$. For two distinct ends $\omega, \eta\in \Omega(D)$ we call the pair  $(\omega,\eta)$ a \emph{limit edge} from $\omega$ to $\eta$ if $D$ has an edge from $C(X,\omega)$ to $C(X,\eta)$ for every finite vertex set $X$ that separates $\omega$ and $\eta$. 

Similarly, for a vertex $v\in V(D)$ and an end $\omega\in \Omega(D)$ we call the pair $(v,\omega)$ a \emph{limit edge} \emph{from} $v$ \emph{to} $\omega$ if $D$ has an edge from $v$ to $C(X,\omega)$ for every finite vertex set $X\subseteq V(D)$ with $v \not\in  C(X, \omega)$. And we call the pair $(\omega,v)$ a \emph{limit edge} \emph{from} $\omega$ \emph{to} $v$ if $D$ has an edge from $C(X,\omega)$ to $v$ for every finite vertex set $X\subseteq V(D)$ with $v \not\in C(X, \omega)$. We write $\Lambda(D)$ for the set of limit edges of $D$.

The digraph in Figure~\ref{fig:ladder} has a limit edge from the lower end to the upper end, and the digraph in Figure~\ref{fig: dominated ray} has a limit edge from the lower vertex to the unique~end. Let us enumerate from left to right the vertical edges  $e_0,e_1,\ldots$ of the digraph $D$ in Figure~\ref{fig:ladder}. We may think of the $e_n$ as converging towards the unique limit edge. This will be made precise in the second paper of our series.

\begin{figure}
\centering
\def\svgwidth{0.8\columnwidth}
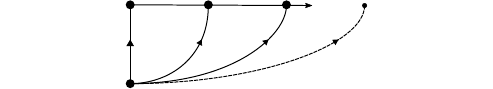
\caption{A digraph with one end (depicted as a small dot) and a limit edge (depicted as a dashed line) from the lower vertex to the end. Every undirected edge in the figure represents a pair of inversely directed edges.} 
\label{fig: dominated ray}
 \end{figure}

Every limit edge $\omega\eta$ between two ends naturally defines a map $f_{\omega\eta}$ with domain $\cX(D)$ as follows. If $X\in \cX(D)$ separates $\omega$ and $\eta$, then $f_{\omega\eta}$ maps $X$ to the set of edges between $C(X,\omega)$ and $C(X,\eta)$; otherwise $f_{\omega\eta}$ maps $X$ to the strong component of $D-X$ in which both ends live. The map $f_{\omega\eta}$ is consistent in that $f_{\omega \eta}(X) \supseteq~f_{\omega\eta}(Y)$ whenever $X \subseteq Y$.\footnote{Here, as later in this context, we do not distinguish rigorously between a strong component and its set of edges. Thus if $Y$ separates $\omega$ and $\eta$ but $X\subseteq Y$ does not, the expression $f_{\omega\eta}(X)\supseteq f_{\omega\eta}(Y)$ means that the strong component $f_{\omega \eta}(X)$ of $D-X$ contains all the edges from the edge set $f_{\omega\eta}(Y)$.}

This gives rise to a second type of direction of a digraph $D$, as follows. Given $X\in \cX(D)$, a non-empty set of edges is a \emph{bundle of $D-X$} if it is the set of all the edges from $C$ to $C'$, or from $v$ to $C$, or from $C$ to $v$, for strong components $C$ and $C'$ of $D-X$ and a vertex $v\in X$.
A \emph{direction} of $D$ is a map $f$ with domain $\cX(D)$ that maps every $X \in \cX(D)$ to a strong component of $D-X$ or to a bundle of $D-X$ so that $f(X) \supseteq f(Y)$ whenever $X \subseteq Y$. We call a direction of $D$ an \emph{edge-direction} of $D$ if there is some $X \in \cX(D)$ such that $f(X)$ is a bundle of $D-X$, in other words, if it is not a vertex-direction. 
Hence $f_{\lambda}$  is an edge-direction for  limit edges $\lambda$ between two ends, and for limit edges  $\lambda$ between vertices and ends an edge-direction $f_{\lambda}$ can be defined analogously. Our next theorem states that every edge-direction can be described in this way:

\begin{customthm}{\ref{thm: edge-direction}} Let $D$ be any infinite digraph. The map $\lambda \mapsto f_{\lambda}$ with domain $\Lambda(D)$ is a bijection between the  limit edges and the  edge-directions of $D$.
\end{customthm}

This paper is organised as follows. In Section~\ref{section: preliminaries} we provide the basic terminology that we use throughout this paper. In Section~\ref{section: ends} we prove the necklace lemma and discuss some basic properties of ends of digraphs. In Section~\ref{section: directions} we prove Theorem~\ref{thm: directions}. Finally, in~Section~\ref{section: limit edges} we investigate limit edges and prove Theorem~\ref{thm: edge-direction}.


\section{Preliminaries}\label{section: preliminaries} 
\noindent Any graph-theoretic notation not explained here can be found in Diestel's textbook~\cite{DiestelBook5}. For the sake of readability, we sometimes omit curly brackets of singletons, i.e., we write $x$ instead of $\{x\}$ for a set~$x$.
Furthermore, we omit the word `directed'---for example in `directed path'---if there is no danger of confusion.

Throughout this paper $D$ is an infinite digraph without multi-edges and without loops, but which may have inversely directed edges between distinct vertices. For a digraph $D$, we write $V(D)$ for the vertex set of $D$, we write $E(D)$ for the edge set of $D$ and $\cX(D)$ for the set of finite vertex sets of $D$. We write edges as ordered pairs $(v,w)$ of vertices $v,w\in V(D)$, and  we usually write $(v,w)$ simply as $vw$. The \emph{reverse} of an edge $vw$ is the edge $wv$. More generally, the \emph{reverse} of a digraph $D$ is the digraph on $V(D)$ where we replace every edge of $D$ by its reverse, i.e., the reverse of $D$ has the edge set $\{\, vw \mid wv \in E(D)\,\}$.
A \emph{symmetric} path is a digraph obtained from an undirected path by replacing each of its edges by its two orientations as separate directed edges. Similarly, a \emph{symmetric ray} is a digraph obtained from an undirected ray by replacing each of its edges by its two orientations as separate directed edges. Hence the reverse of any symmetric path or symmetric ray is a symmetric path or symmetric ray, respectively.

The directed subrays of a ray are its \emph{tails}. Call a ray \emph{solid} in $D$ if it has a tail in some strong component of $D-X$ for every finite vertex set $X\subseteq V(D)$.

Two solid rays in $D$ are \emph{equivalent}, if they have a tail in the same strong component of $D-X$ for every finite vertex set $X \subseteq V(D)$. We call the equivalence classes of this relation the \emph{ends} of $D$ and we write $\Omega(D)$ for the set of ends of $D$.

Similarly, the reverse subrays of a reverse ray are its \emph{tails}. We call a reverse ray \emph{solid} in $D$ if it has a tail in some strong component of $D-X$ for every finite vertex set $X\subseteq V(D)$. With a slight abuse of notation, we say that a reverse ray $R$ \emph{represents} an end $\omega$ if there is a solid ray $R'$ in $D$ that represents $\omega$ such that $R$ and $R'$ have a tail in the same strong component of $D-X$ for every finite vertex set $X\subseteq V(D)$.

For a finite vertex set $X \subseteq V(D)$ and a strong component $C$ of $D-X$ an end $\omega$ is said to \emph{live in} $C$ if one (equivalent every) solid ray in $D$ that represents $\omega$ has a tail in $C$. We write $C(X,\omega)$ for the strong component of $D-X$ in which $\omega$ lives. For two ends $\omega$ and $\eta$ of $D$ a finite set $X \subseteq V(D)$ is said to \emph{separate} $\omega$ and $\eta$ if $C(X, \omega) \neq C(X, \eta )$, i.e., if $\omega$ and $\eta$ live in distinct strong components  of $D-X$.

Given sets $A, B\subseteq V(D)$ of vertices a \emph{path from $A$ to $B$}, or \emph{$A$--$B$} path is a path that meets $A$ precisely in its first vertex and $B$ precisely in its last vertex. We say that a vertex $v$ can \emph{reach} a vertex $w$ in $D$  if there is a $v$--$w$ path in $D$. A set $W$ of vertices is \emph{strongly connected} in $D$ if every vertex of $W$ can reach every other vertex of $W$ in $D[W]$. 

Let $H$ be any fixed digraph. A \emph{subdivision} of $H$ is any digraph that is obtained from $H$ by replacing every edge $vw$ of $H$ by a path $P_{vw}$ with first vertex $v$ and last vertex $w$ so that the paths $P_{vw}$ are  internally disjoint and do not meet~$V(H) \setminus \{v,w\}$. We call the paths $P_{vw}$ \emph{subdividing} paths. If $D$ is a subdivision of $H$, then the original vertices of $H$ are the \emph{branch vertices} of $D$ and the new vertices its \emph{subdividing vertices}. 

\emph{An inflated} $H$ is any digraph that arises from a subdivision $H'$ of $H$ as follows. Replace every branch vertex $v$ of $H'$ by a strongly connected digraph $H_v$ so that the $H_v$ are disjoint and do not meet any subdividing vertex; here replacing means that we first delete $v$ from $H'$ and then add $V(H_v)$ to the vertex set and $E(H_v)$ to the edge set. Then replace every subdividing path $P_{vw}$ that starts in $v$ and ends in $w$ by an $H_v$--$H_w$ path that coincides with $P_{vw}$ on inner vertices. We call the vertex sets $V(H_v)$ the \emph{branch sets} of the inflated $H$. A \emph{necklace} is an inflated symmetric ray with finite branch sets; the branch sets of a necklace are its \emph{beads}. (See Figure~\ref{fig: necklace} for an example of a necklace.)

A vertex set $Y \subseteq V(D)$ \emph{separates} $A$ and $B$ in $D$ with $A,B\subseteq V(D)$ if every $A$--$B$ path meets $Y$, or if every $B$--$A$ path meets $Y$. For two vertices $v$ and $w$ of $D$ we say that $Y\subseteq V(D) \setminus \{v,w\}$ \emph{separates} $v$ and $w$ in $D$, if it separates $\{v\}$ and $\{w\}$ in $D$. A \emph{separation} of $D$ is an ordered pair $(A,B)$ of vertex sets $A$ and $B$ with $V(D) =A \cup B$ for which there is no edge from $B\setminus A$ to $A\setminus B$. The set $A\cap B$ is the  \emph{separator} of $(A,B)$ and the vertex sets $A$ and $B$ are the two \emph{sides} of the separation $(A,B)$. Note that the separator of a separation indeed separates its two sides. The size of the separator of a separation $(A,B)$ is the \emph{order} of $(A,B)$. Separations of finite order are also called \emph{finite order separations}. There is a natural way to compare separations, namely one defines $(A_1,B_1)\le (A_2,B_2)$ if $A_1\subseteq A_2$ and $B_2\subseteq B_1$. Regarding to this partial order $(A_1\cup A_2, B_1\cap B_2)$ is the supremum and $(A_1\cap A_2,B_1\cup B_2)$ is the infimum of two separations $(A_1,B_1)$ and $(A_2,B_2)$. More generally, if $((A_i,B_i))_{i\in I}$ is a family of separations, then 
$$(\bigcup_{i\in I}A_i , \bigcap_{i\in I}B_i)\;\textup{ and }\;(\bigcap_{i\in I} A_i, \bigcup_{i\in I}B_i)$$
 is its supremum and infimum, respectively.

For vertex sets $A, B\subseteq V(D)$ let $E(A,B)$ be the set of edges from $A$ to $B$, i.e., $E(A,B)= (A\times B)\cap E(D)$. Given a subdigraph $H\subseteq D$, a \emph{bundle} of $H$ is a non-empty edge set of the form $E(C,C')$, $E(v,C)$, or $E(C,v)$ for strong components $C$ and $C'$ of $H$ and a vertex $v\in V(D) \setminus V(H)$.
We say that $E(C,C')$ is a bundle, \emph{between strong components} and $E(v,C)$ and $E(C,v)$ are bundles \emph{between a vertex and a strong component.} In this paper we consider only bundles of subdigraphs $H$ with $H=D-X$ for some $X\in \cX(D)$.


Now, consider a vertex $v\in V(D)$, two ends $\omega, \eta\in \Omega(D)$ and a finite vertex set $X\subseteq V(D)$. If $X$ separates $\omega$ and $\eta$ we write $E(X,\omega\eta)$ as short for $E(C(X,\omega),C(X,\eta))$. Similarly, if $v \in C'$ for a strong component $C' \neq C(X,\omega)$ of $D-X$ we write $E(X,v\omega)$ and $E(X,\omega v)$ as short for the edge set $E(C', C(X,\omega))$ and $E(C(X,\omega),C')$, respectively. If $v \in X$  we write $E(X,v\omega)$ and $E(X,\omega v)$ as short for $E(v, C(X,\omega))$ and $E(C(X,\omega),v)$, respectively.  Note that $E(X,\omega\eta)$, $E(X,v\omega)$ and $E(X,\omega v)$ each are bundles if they are non-empty.

An \emph{arborescence} is a rooted oriented tree $T$ that contains for every vertex \mbox{$v \in V(T)$} a directed path from the root to $v$. 
The vertices of any arborescence are partially ordered as $v \leq_T w$ if $T$ contains a directed path from $v$ to $w$. We write $\uc{v}_T$ for the up-closure of $v$ in $T$.

A \emph{directed star} is an arborescence whose underlying tree is an undirected star that is centred in the root of the arborescence. A \emph{directed comb} is the union of a ray with infinitely many finite disjoint paths (possibly trivial) that have precisely their first vertex on $R$. Hence the underlying graph of a directed comb is an undirected comb. The \emph{teeth} of a directed comb or reverse directed comb are the teeth of the underlying comb. The ray from the definition of a comb is the \emph{spine} of the comb.

\section{Necklace Lemma}\label{section: ends} 
\noindent This section is dedicated to the necklace lemma. We begin with our directed version of the star-comb lemma, which motivates the necklace lemma. Then we continue with the definition of the $U$-rank, in fact we will define the $U$-rank in a slightly more general setting by considering not only one set $U$ but finitely many. Finally, we prove the necklace lemma and provide two of its applications.

The star-comb lemma~\cite{DiestelBook5} for undirected graphs is a standard tool in infinite graph theory and reads as follows:
\begin{lemma}[Star-Comb Lemma]\label{lemma: star-comb}
Let $U$ be an infinite set of vertices in a connected undirected graph $G$.
Then $G$ contains a comb 
with all its teeth in $U$ or a subdivided infinite star with all its leaves in $U$.
\end{lemma} 
Let us see how to translate the star-comb lemma to digraphs. 
Given a set $U$ of vertices in a digraph, a comb \emph{attached} to $U$ is a comb with all its teeth in $U$ and a star \emph{attached} to $U$ is a subdivided infinite star with all its leaves in $U$.
The set of teeth is the \emph{attachment set} of the comb and the set of leaves is the \emph{attachment set} of the star.
We adapt the notions of `attached to' and `attachment sets' to reverse combs or reverse stars, respectively.

\begin{lemma}[Directed Star-Comb Lemma]
Let $D$ be any strongly connected digraph and let $U\subseteq V(D)$ be infinite. Then $D$ contains a star or comb attached to $U$ and a reverse star or reverse comb attached to $U$ sharing their attachment sets.
\end{lemma}

\begin{proof}
Since $D$ is strongly connected we find a spanning arborescence $T$.   Applying the star-comb lemma in the undirected tree underlying $T$ yields a comb or a star attached to $U$ (without loss of generality the spine starts in the root of $T$). Let $U'$ be the attachment set in either case. 

Again using that $D$ is strongly connected we find a reverse spanning arborescence~$T'$. Applying the star-comb lemma a second time, now in the undirected tree underlying $T'$ yields a reverse comb or a reverse star attached to $U'$. Thinning out the teeth or leaves of the  comb or  star, respectively, completes the proof. 
\end{proof}
The star-comb lemma fundamentally describes how an infinite set of vertices can be connected in an infinite graph, namely through stars and combs.  Similarly, the directed  star-comb lemma describes the nature of strong connectedness in infinite digraphs. Indeed, adding a single path from the first vertex  of the reverse comb's spine or centre of the reverse star to the first vertex  of the comb's spine or centre of the star, respectively, yields a strongly connected digraph that intersects $U$ infinitely. We shall use the directed star-comb lemma in the proof of one of our main results in the second part of this series~\cite{EndsOfDigraphsII}.

As noted in the introduction the star-comb lemma is often used in order to find an end of a given undirected graph $G$ in the closure of an infinite set $U\subseteq V(G)$ of vertices. This is usually done in situations where $G$ contains no infinite subdivided star with all its leaves in $U$; for example if the graph is locally finite. Then the star-comb lemma in $G$ applied to $U$ always returns a comb with all its teeth in $U$ and the end represented by the comb's spine is contained in the closure of $U$.  

The directed star-comb lemma however does not manage the task of finding an end of a digraph in the closure of an infinite set of vertices. Consider for example the digraph $D$ that is obtained from the digraph in Figure~\ref{fig:ladder} by subdividing each vertical edge once. We write $U$ for the set of subdividing vertices. As $D$ contains neither an infinite star nor an infinite reverse star, the directed star-comb lemma applied to $U$ returns a comb attached to $U$ and a reverse comb attached to $U$ sharing their attachment sets. Therefore we would expect that the ends that are represented by the spines are contained in the closure of $U$. But $U$ does not have any end in its closure because the subdividing vertices all lie in singleton strong components of $D$.

The necklace lemma will perform the task of finding and end in the closure of a given set of vertices. Before we state it, we need to introduce the $\mathcal{U}$-rank for digraphs: For this, consider a finite set $\cU$ and think of $\cU$ as consisting of infinite sets of vertices. We define in a transfinite recursion the class of digraphs that have a $\cU$-rank. A digraph $D$ has \emph{$\cU$-rank} $0$ if there is a set $U\in\cU$ such that $U\cap V(D)$ is finite. It has \emph{$\cU$-rank} $\alpha$ if it has no $\cU$-rank $<\alpha$ and there is some $X\in\cX(D)$ such that every strong component of $D-X$ has a $\cU$-rank $<\alpha$.  In the case  $U=V(D)$ we call the $U$-rank of $D$  the \emph{rank of $D$} (provided that $D$ has a $U$-rank). Note that  if $U \supseteq V(D)$ for a digraph, then its $U$-rank equals its rank.

We remark that our notion of ranking extends the notion of Schmidt's \emph{ranking of rayless graphs}, in that the rank of a given undirected graph $G$ is precisely the rank of the digraph obtained from $G$ by replacing every edge by its two orientations as separate directed edges, see \cite{Schmidt1983} or Chapter~8.5 of~\cite{DiestelBook5} for Schmidt's rank. More generally, for a  set $U$, our $U$-rank of digraphs extends the notion of the $U$-rank of graphs, in that an undirected graph $G$ has a $U$-rank if and only if the digraph that is obtained from $G$ by replacing every edge by its two orientations as separate directed edges has a $U$-rank; see \cite{StarComb1StarsAndCombs} for the definition of the $U$-rank of an undirected graph. 

Before we prove the necklace lemma, we provide two basic lemmas for the $\cU$-rank of digraphs:
\begin{lemma}\label{lemma: rank and subgraph}
Let $D$ be a digraph and let $\cU$ be a finite set. 
If $D$ has $\cU$-rank $\alpha$ and $H\subseteq D$, then $H$ has some $\cU$-rank $\le\alpha$. 
\end{lemma}
\begin{proof}
We prove the statement by transfinite induction on the $\cU$-rank of $D$. Clearly, if $D$ has  $\cU$-rank $0$, then so does every subdigraph. Let $D$ be a digraph with $\cU$-rank $\alpha$ and $H\subseteq D$. We find a finite vertex set $X\subseteq V(D)$ such that every strong component of $D-X$ has $\mathcal{U}$-rank less than $\alpha$. As every strong component of $H-X$ is contained in a strong component of $D-X$, every strong component of $H-X$ has a $\cU$-rank less than $\alpha$ by the induction hypothesis. Hence $H$ has a $\cU$-rank $\leq \alpha$
\end{proof}

\newpage

\begin{lemma}\label{lemma: rank inf many comps meet U for witness}
Let $D$ be any digraph and let $\cU$ be a finite set. If $D$ has a $\cU$-rank $\alpha>0$ and $X\subseteq V(D)$ is a finite vertex set such that every strong component of $D-X$ has a $\cU$-rank  $<\alpha$, then infinitely many strong components of $D-X$ meet every set in $\cU$.
\end{lemma}

\begin{proof}
Suppose for a contradiction that the set $\mathcal{C}$ of strong components of $D-X$ that meet every set in $\cU$ is finite. We find for every $C\in \mathcal{C}$ a finite vertex set $X_C\subseteq V(C)$ witnessing that $C$ has a $\cU$-rank $<\alpha$. Let $Y$ be the union of $X$ with all the finite vertex sets $X_C$. Then $Y$ witnesses that $D$ has a $\cU$-rank $<\alpha$ contradicting our assumption that $D$ has $\cU$-rank $\alpha$. 
\end{proof}

Given a set $\cU$, a necklace $N\subseteq D$ is \emph{attached to} $\cU$ if infinitely many beads of $N$ meet every set in $\cU$. 
\begin{mainlemma}[Necklace Lemma]\label{lemma: necklace lemma}
Let $D$ be any digraph and  $\cU$ a finite set of vertex sets of $D$. Then the following statements are complementary: \begin{enumerate}
    \item $D$ has a necklace attached to $\cU$;    
    \item $D$ has a $\cU$-rank. 
\end{enumerate}
\end{mainlemma}

\begin{proof}
Let us start by showing that not both statements hold at the same time. Suppose for a contradiction there is a digraph $D$ that has a $\cU$-rank and contains a necklace attached to $\cU$ as a subdigraph. Then, by Lemma~\ref{lemma: rank and subgraph}, every necklace $N \subseteq D$ has a $\cU$-rank. But deleting finitely many vertices from any necklace attached to $\cU$ leaves a strong component that is a necklace attached to $\cU$ by its own. Hence choosing a necklace $N \subseteq D$ attached to $\cU$ with minimal $\cU$-rank results in a contradiction.

In order to prove that at least one of (i) and (ii) holds, let us assume that $D$ has no $\cU$-rank. Then for every $X\in \cX(D)$, the digraph $D-X$ has a strong component that has no $\cU$-rank. In particular, every such strong component contains a vertex---in fact infinitely many---from every set in $\cU$.

We will recursively construct an ascending sequence $(H_{n})_{n \in \N}$ of inflated symmetric paths with finite branch sets, so that $H_n$ extends $H_{n-1}$, by adding an inflated vertex $Y_n$ that meets every set in $\cU$. In order to make the construction work, we will make sure that $Y_n$ is contained in a strong component of $D-X_n$ that has no $\cU$-rank, where $X_n=H_n\setminus Y_n$. The overall union of the $H_{n}$ then gives a necklace attached to $\cU$.

Let $H_0=Y_0$ be a finite strongly connected vertex set that is included in a strong component of $D=D-\emptyset$, that has no $\cU$-rank, and that meets every set in~$\cU$.
Now, suppose that $n\in \N$ and that $H_n$ and $Y_n$ have already been defined.  Let $C$ be the strong component of $D-X_{n}$ that includes $Y_n$. As $C$ has no $\cU$-rank, the digraph $C-Y_n$ has a strong component $C'$ that has no $\cU$-rank. Let $P$ be a path in $C$ from $Y_n$ to $C'$ and $Q$ a path from $C'$ to $Y_n$. Note that $P$ and $Q$ are internally disjoint.
Let $Y_{n+1}\subseteq C'$ be a strongly connected vertex set that contains the last vertex of $P$, the first vertex of $Q$ and one vertex of every set in $\cU$. 
We define $H_{n+1}$ to be the union of $H_n$, $P$, $Q$ and $Y_{n+1}$.
\end{proof}

As our first application of the necklace lemma we describe the connection between Zuther's notion of ends  from \cite{Zuther96Thesis}, which we call pre-ends, with our notion of ends. Two rays or reverse rays $R_1,R_2\subseteq D$ are \emph{equivalent}, if there are infinitely many  disjoint paths from $R_1$ to $R_2$ and infinitely many  disjoint paths from $R_2$ to $R_1$. We call the equivalence classes of this relation the \emph{pre-ends} of $D$.

\begin{lemma}\label{lemma: characterisation end}
Let $D$ be any digraph and $\gamma$ a pre-end of $D$. Then $\gamma$ includes an end $\omega$ of $D$ if and only if $\gamma$ is represented both by a ray and a reverse ray. Moreover, $\omega$ is the unique end of $D$ included in $\gamma$. 
\end{lemma}

\begin{proof}
Consider any pre-end $\gamma$ of $D$. For the forward implication suppose that $\gamma$ includes an end $\omega$ of $D$. Then there is a ray  $R$ that is solid in $D$ and that represents~$\gamma$. It suffices to find a necklace that is attached to $U:= V(R)$. Indeed, every necklace $N$ contains a ray and a reverse ray and if $N$ is attached to $R$ then these rays must be equivalent to $R$. 

So suppose for a contradiction that there is no such necklace. Then by the necklace lemma applied to $U$ in $D$, the digraph $D$ has a $U$-rank, say $\alpha$. Let $X\subseteq V(D)$ be a finite vertex set that witnesses that the $U$-rank of $D$ is $\alpha$. As $U\subseteq V(D)$ is infinite, we have $\alpha>0$. Now, it follows by Lemma~\ref{lemma: rank inf many comps meet U for witness} that the ray $R$ meets infinitely many strong components of $D-X$. We conclude that $R$ has no tail in any strong component of $D-X$ contradicting that $R$ is solid in $D$.

For the backward implication we assume that $\gamma$ is represented by a ray and a reverse ray. We prove that every ray $R$ that represents $\gamma$ is solid in $D$. So let $R$ be any ray that represents $\gamma$ and let $R'$ be a reverse ray that represents $\gamma$. As $R$ and $R'$ are equivalent we find a path system $\mathcal{P}$ that consists of infinitely many pairwise disjoint paths from $R$ to $R'$ and infinitely many pairwise disjoint paths from $R'$ to~$R$.

The subdigraph $H$ of $D$ that consists of $R$, $R'$ and all the paths in $\mathcal{P}$  has exactly one infinite strong component and finitely many finite strong components (possibly none). Moreover, deleting finitely many vertices from $H$ results again in exactly one infinite strong component and finitely many finite strong components. Consequently, $R$ has a tail that is contained in a strong component of $D-X$ for every finite vertex set $X \subseteq V(D)$.

For the `moreover' part  note that the above argument shows that any ray that represents $\gamma$ has a tail in the same strong component of $D-X$ as the reverse ray $R'$,  for every finite vertex set $X \subseteq V(D)$. Consequently, any two rays that represent $\gamma$ have a tail in the same strong component of $D-X$ for every finite vertex set $X \subseteq V(D)$.
\end{proof}

Our second application of the necklace lemma demonstrates how the rank can be used to prove statements about digraphs that have no end. A set of vertices of a digraph $D$ is \emph{acyclic} in $D$ if its induced subdigraph does not contain a directed cycle. The \emph{dichromatic number} \cite{neumanndichromatic} of a digraph $D$ is the smallest cardinal $\kappa$ so that $D$ admits a vertex partition into $\kappa$ partition classes that are acyclic in~$D$.  As a consequence of the necklace lemma we obtain a sufficient condition for $D$ to have a countable dichromatic number: 
\begin{theorem} If $D$ is a digraph that contains no necklace as a subdigraph, then the dichromatic number of $D$ is countable. 
\end{theorem} 
\begin{proof}
By the necklace lemma, the statement that $D$ contains no necklace as a subdigraph is equivalent to the statement that $D$ has a rank. Therefore we can apply induction on the rank of $D$. The vertex set of a finite digraph clearly has a partition into finitely many singleton---and thus acyclic---partition classes, which settles the base case. Now assume that $D$ has rank $\alpha>0$ and that the statement is true for all ordinals $<\alpha$. We find a finite vertex set $X\subseteq V(D)$ such that every strong component of $D-X$ has some rank~$<\alpha$. Hence the induction hypothesis yields a partition $\{\, V_{n}(C)\mid n\in \N\, \}$ of every strong component $C$ of $D-X$ into acyclic partition classes. For every $n\in\N$, let $V_n$ consist of the union of all the sets $V_n(C)$ with $C$ a strong component of $D-X$. Note that $V_n$ is acyclic in $D$. Combining a partition of $X$ into singleton partition classes with the partition $\{\, V_n\mid n\in\N\,\}$ of $V(D-X)$ completes the induction step.  
\end{proof}

\section{Directions}\label{section: directions} 
\noindent In this section we will prove our main result. To state it properly we need two definitions. A direction of a digraph $D$ is a map $f$ with domain $\cX(D)$ that sends every $X \in \cX(D)$ to a strong component or a bundle of $D-X$ so that $f(X)\supseteq f(Y)$ whenever $X\subseteq Y$. We call a direction $f$ of $D$ a \emph{vertex-direction} if $f(X)$ is a strong component of $D-X$ for every $X \in \cX(D)$. 

Every end of $D$ naturally defines a direction $f_\omega$ which maps every finite vertex set $X\subseteq V(D)$ to the unique strong component of $D-X$ in which every ray that represents $\omega$ has a tail. Now, our first main theorem reads as follows: 

\begin{mainresult}\label{thm: directions} 
Let $D$ be any infinite digraph. The map $\omega \mapsto f_\omega$ with domain $\Omega(D)$ is a bijection between the ends and the vertex-directions of $D$. 
\end{mainresult} 
\noindent The proof of this needs some preparation. Let $D$ be any digraph and let $\cU$ be a set of vertex sets of $D$. We say that an end $\omega$ of $D$ is contained in the  \emph{closure} of $\cU$ if $C(X,\omega)$ meets every vertex set in $U\in\cU$ for every finite vertex set $X\subseteq V(D)$. In the second paper~\cite{EndsOfDigraphsII} of this series we will define a topology on the space $|D|$ formed by $D$ together with its ends and limit edges and in this topology an end $\omega$ will be in the closure of $\cU$ if and only if it is in the topological closure of every set in $\cU$. Note that an end $\omega$ is contained in the closure of the vertex set of a ray $R$ if and only if $R$ represents $\omega$.

Similarly, we say that a vertex-direction $f$ of $D$ is contained in the \emph{closure} of~$\cU$, if $f(X)$ meets every $U\in\cU$ for every $X\in \cX(D)$. Note that if $f$ is contained in the closure of $\cU$, then $f(X)$ meets every $U\in \cU$ in an infinite vertex set. The following lemma describes the connection between ends in the closure of $\cU$, vertex-directions in the closure of $\cU$ and necklaces attached to $\cU$:

\begin{lemma} \label{lemma: connection closure; direction; necklace} Let $D$ be any digraph, and let $\cU$ be a finite set of vertex sets of $D$. Then the following assertions are equivalent:
\begin{enumerate}
    \item $D$ has an end in the closure of $\cU$;
    \item $D$ has a vertex-direction in the closure of $\cU$;
    \item $D$ has a necklace attached to $\cU$.
\end{enumerate}
\end{lemma}
 
\begin{proof}  (i)$\rightarrow$(ii): Let $\omega$ be any end in the closure of $\cU$. It is straightforward to check that $f_\omega$ is a vertex-direction in the closure of $\cU$. 

(ii)$\rightarrow$(iii): Suppose that $f$ is a vertex-direction in the closure of $\cU$. We need to find a necklace attached to $\cU$. By the necklace lemma we may equivalently show that $D$ has no $\cU$-rank. Suppose for a contradiction that $D$ has a $\cU$-rank $\alpha$. By Lemma~\ref{lemma: rank and subgraph} subdigraphs of digraphs that have a $\cU$-rank have a $\cU$-rank and thus we may choose $X'$ such that $f(X')$ has the smallest $\cU$-rank among all $f(X)$ with $X\in \cX(D)$. Note that $f(X)$ has $\cU$-rank $\geq 1$ for every $X\in \cX(D)$. Indeed, if $f(X)\cap U$ is finite for some $U\in \cU$, then $$f(X\cup (f(X)\cap U))\cap U= \emptyset$$ contradicting that $f$ is a vertex-direction in the closure of $\cU$. Hence we find a finite vertex set $X^{\prime \prime}\subseteq f(X^\prime)$ such that all strong components of $f(X^\prime) - X^{\prime \prime}$ have $\cU$-rank less than that of $f(X^\prime)$. But then $X'\cup X''$ would have been a better choice for $X'$.

(iii)$\rightarrow$(i): Given a necklace $N$ attached to $U$, let $R\subseteq N$ be a ray. Then $R$ is solid in $D$. It is straightforward to show that the end  that is represented by $R$ is contained in the closure of~$\cU$. 
\end{proof}

Let $D$ be any digraph and let $f$ be any vertex-direction of $D$. We think of a separation $(A,B)$ of $D$ as pointing towards its side $B$. Now, if $(A,B)$ is a finite order separation of $D$, then $f(A\cap B)$ is either included in $B\setminus A$ or $A\setminus B$. In the first case we say that $(A,B)$ \emph{points towards} $f$ and in the second case we say that $(A,B)$ \emph{points away from} $f$. Note that the supremum or infimum of two finite order separations is again a finite order separation. If two separations  point towards or away from $f$, then the same is true for their supremum or infimum, respectively:

\begin{lemma}\label{lemma: separations pointing to f and supremum} Let $D$ be any digraph and let $f$ be a vertex-direction of $D$. Suppose that  $(A_1,B_1)$ and $(A_2,B_2)$ are finite order separations of $D$.  \begin{enumerate}
    \item If $(A_1,B_1)$ and $(A_2,B_2)$ point towards $f$, then $(A_1\cup A_2,B_1\cap B_2)$ points towards $f$.
    \item If $(A_1,B_1)$ and $(A_2,B_2)$ point away from $f$, then $(A_1\cap A_2,B_1\cup B_2)$ points away from $f$.
\end{enumerate}
\end{lemma} 

\begin{proof} (i) We have to show that for $(A,B):=(A_1\cup A_2,B_1\cap B_2)$ and $X=A\cap B$ the strong component $f(X)$ is included in $B\setminus A$. For this let us consider the auxiliary separation $(A,B'):=(A,X'\cup B)$, where $X':= \bigcup_{i=1,2} A_i\cap B_i$ (cf. Figure~\ref{fig:separations}). Recall that the separator of a separation separates its two sides. Hence the vertex set $B\setminus A$ is partitioned into the strong components of $D-X$ that it meets. 

\begin{figure}[ht]
\centering
\def\svgwidth{7cm}
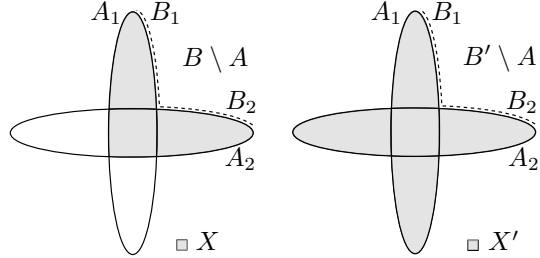
\caption{The separations $(A,B)$ and $(A,B')$  from the proof of Lemma~\ref{lemma: separations pointing to f and supremum}.}
\label{fig:separations}
\end{figure}

First, we observe that $(A,B')$ points towards $f$, a fact that we verify as follows: Since $A_i\cap B_i\subseteq X'$ and because $(A_i,B_i)$ points towards $f$ we have $$f(X')\subseteq f(A_i\cap B_i)\subseteq  B_i$$ for $i=1,2$. Hence $f(X')\subseteq B_1\cap B_2=B$. Now, $f(X')$ avoids $X'$ because it is a strong component of $D-X'$, giving $f(X')\subseteq B\setminus X'$. As $B\setminus X'=B'\setminus A$, we conclude that $f(X')$ is included in $B'\setminus A$.

Second, we observe that the strong components of $D-X$ that partition $B\setminus A$ are exactly the strong components of $D-X'$ that meet $B'\setminus A$; the reason for this is that $B'$ is obtained from $B$ by adding only vertices from $A\setminus B$.

Finally, we employ the two observation in order to prove that $(A,B)$ points towards $f$. Indeed, since $X\subseteq X'$ and because $f$ is a vertex-direction, we have that $f(X')\subseteq f(X)$.  Now, the first observation says that $f(X')$ is included in $B'\setminus A$. Together with the second observation we obtain $f(X')=f(X)$. So the equation $B'\setminus A=B\setminus A$ yields $f(X)\subseteq B\setminus A$ as desired. 

(ii) Apply (i) to the reverse of $D$.
\end{proof}

Recall, that for a given undirected graph $G$ a vertex $v$ is said to \emph{dominate} an end $\omega$ of $G$ if there is an infinite $v$--$R$ fan in $G$ for some (equivalently every) ray~$R$ that represents $\omega$. Equivalently $v$ dominates $\omega$ if $v$ is contained in $ C(X,\omega)$ for every finite vertex set $X\subseteq V(G)\setminus\{v\}$. An end $\omega\in \Omega(G)$ is \emph{dominated} if some vertex of $G$ dominates it.
Ends not dominated by any vertex of $G$ are \emph{undominated}, see~\cite{DiestelBook5}. The main case distinction in the proof of Diestel and K\"uhn's theorem \cite[Theorem~2.2]{Diestel_Kuehn_Directions_03}, which states that the ends of an undirected graph correspond bijectively to its directions, essentially distinguishes between directions that correspond to dominated ends and those that correspond to undominated ends. Our plan is to make a similar case distinction for which we need a concept of domination for ends of digraphs.

Let $D$ be any digraph. For a vertex $a\in V(D)$ and $B\subseteq V(D)$ a set of $a$--$B$ paths in $D$ is called an $a$--$B$ \emph{fan} if any two of the paths meet precisely in $a$. Similarly, a set of $B$--$a$ paths in $D$ is called an $a$--$B$ \emph{reverse fan} if any two of the paths meet precisely in $a$.  We say that a vertex $v\in V(D)$ \emph{ dominates} a ray $R\subseteq D$ if there is an infinite  $v$--$R$ fan in $D$. The vertex $v$ dominates an end $\omega \in \Omega(D)$ if it dominates some (equivalently every) ray that represents $\omega$.  Similarly, a vertex $v\in V(D)$ \emph{reverse dominates} a ray $R\subseteq D$ if $D$ contains a $v$--$R$ reverse fan. The vertex $v$ \emph{reverse dominates} an end $\omega \in \Omega(D)$ if it reverse dominates some (equivalently every) ray that represents $\omega$. An end of $D$ is \emph{dominated} or \emph{reverse} \emph{dominated} if some vertex dominates or reverse dominates it, respectively. 

Now, we translate the concept of domination and reverse domination to vertex-directions of digraphs. A vertex $v\in V(D)$ \emph{dominates} a vertex-direction $f$ in $D$, if $v\in A$ for every finite order separation $(A,B)$ of $D$ that points away from $f$. If $f$ is dominated by some vertex, then it is \emph{ dominated}. Similarly, $v$ \emph{reverse dominates} $f$ if $v\in B$ for every finite order separation $(A,B)$ of $D$ that points towards $f$. If $f$ is reverse dominated by some vertex, then $f$ is \emph{reverse dominated}. The following proposition shows that our translation of the concept forwards and reverse domination to vertex-directions of digraphs is accurate:
\begin{proposition}
Let $D$ be any digraph and $\omega$ an end of $D$. A vertex (reverse) dominates $\omega$ if and only if it (reverse) dominates~$f_\omega$.
\end{proposition}
\begin{proof}
We prove the statement in its `dominates' version; for the `reverse dominates' version consider the reverse of $D$. First, suppose that $v\in V(D)$  dominates $\omega$ and let $(A,B)$ be a finite order separation pointing away from $f_\omega$. Every ray $R$ that represents $\omega$ has a tail in $f_\omega(A\cap B)$; in particular in $D[A]$. As $D$ contains an infinite $v$--$R$ fan and the separator of $(A,B)$ is finite, it follows that $v$ is contained in $A$ as~well.

For the backward implication suppose that $v\in V(D)$ dominates~$f_\omega$. Given a ray $R$ that represents $\omega$, with $v\notin R$ say,  we need to find an infinite $v$--$R$ fan in~$D$. For this, we show that every finite $v$--$R$ fan $F$ in $D$ can be extended by one additional $v$--$R$ path; then an infinite such fan can be constructed recursively in countably many steps. Let $H$ be the union of the paths in $F$ and let $X$ consist of $V(H-v)$ together with the vertices of some finite initial segment of $R$ that contains all the vertices that $H$ meets on $R$. We may view the strong components of $D-X$ partially ordered by $C_1\le C_2$ if there is a path in $D-X$ from $C_1$ to $C_2$. Let $C$ be the strong component of $D-X$ that contains $v$ and let $\uc{C}$ be the set of all the strong components of $D-X$ that are $\ge C$. If $C(X,\omega)$ is contained in $\uc{C}$, then it is easy to find a $v$--$R$ path in $D$ that extends our fan $F$. We claim that this is always the case: Otherwise consider the finite order separation $(A,B)$ with $A:=V(D)\setminus \bigcup \uc{C}$ and $B:=X\cup \bigcup \uc{C}$. On the one hand, $(A,B)$ points away from~$f_\omega$. On the other hand, we have $v\notin A$ contracting that $v$ dominates~$f_\omega$.   
\end{proof}
 
\begin{lemma}\label{lemma: sequence of separations}
Let $D$ be any strongly connected digraph and let $f$ be any  vertex-direction of $D$. Then the following assertions are complementary:
\begin{enumerate}
    \item $f$ is (reverse) dominated;
    \item there is a strictly descending (ascending) sequence $((A_i,B_i))_{i\in\N}$ of finite order separations in $D$ with pairwise disjoint separators all pointing away from (towards) $f$.
\end{enumerate}
Moreover, a vertex-direction 
$f$ as in \emph{(ii)} is the unique vertex-direction in the closure of $U$ for any vertex set $U$ consisting of one vertex of $f(A_i\cap B_i)$ for every $i\in\N$. 
\end{lemma} 

\begin{proof} We prove the case where $f$ is dominated and that the sequence in (ii) is descending; the proof of the case where $f$ is reverse dominated and the sequence in (ii) is ascending can then be obtained by considering the reverse of $D$. To begin, we will show that not both assertions can hold at the same time. Suppose that $((A_i,B_i))_{i\in\N}$ is as in (ii). We show that for every $v\in V(D)$ there is a separation $(A,B)$ of $D$ pointing away from $f$ with $v\in B\setminus A$. We claim that $(A,B):=(A_j,B_j)$ can be taken for $j\in\N$ large enough, a fact that we verify as follows: 

As $D$ is strongly connected there is path from $B_0 \setminus A_0$   to $v$. Let $j$ be the length of a shortest path $P$ from $B_0 \setminus A_0$ to $v$. Then $v$ is contained in $B_j\setminus A_j$, because otherwise $P$ would contain $j+1$ vertices---one from each of the separators $B_i \cap A_i$ with $i\le j$---contradicting that $P$ has length $\le j$.

Next, we assume that $f$ is not dominated  and  construct a sequence $((A_i,B_i))_{i\in\N}$ as in (ii). Let $(A_0,B_0)$ be any finite order separation with non-empty separator pointing away from $f$. To see that such a separation exist consider any non-empty finite vertex set $X\subseteq V(D)$. We may view the strong components of $D-X$ partially ordered by $C_1\le C_2$ if there is a path in $D-X$ from $C_1$ to $C_2$. Let $\dc{f(X)}$ be the down-closure of all the strong components $\le f(X)$. Then we can take $A_0:= \dc{f(X)}\cup X$ and $B_0:= V(D)\setminus \dc{f(X)}$.

Now, assume that $(A_n,B_n)$ has already been defined. Since no vertex  dominates $f$ we find for every $ x \in A_n \cap B_n$ a separation $(A_x,B_x)$ pointing away from $f$ such that $x\in B_x\setminus A_x$. Letting $(A_{n+1},B_{n+1})$ be the infimum of all the $(A_x,B_x)$ and $(A_n,B_n)$ completes the construction. Indeed, $(A_{n+1},B_{n+1})$ points away from $f$ by Lemma~\ref{lemma: separations pointing to f and supremum} and its separator is disjoint from all the previous ones as $A_{n+1} \cap B_{n+1} \subseteq A_n \setminus B_n$.

For the `moreover' part let us write $X_i:= A_i\cap B_i$ for every $i\in\N$. We first show that $f$ is a vertex-direction in the closure of $U$. Given $X\in \cX(D)$ we need to show that $f(X)$ meets~$U$. With a distance argument as above one finds $j$ such that all the vertices of $X$ are contained in $B_j\setminus A_j$. Then $f(X_j)$ is included in $f(X)$ because $f(X_j)=f(X\cup X_j)$. In particular $f(X)$ contains the vertex from $U$ that was picked from~$f(X_j)$.

Finally, we prove that $f=f'$ for every vertex-direction $f'$ that is in the closure of $U$. Given $f'$ it suffices to show that $f(X_i)=f'(X_i)$ for every $i\in \N$: then $f(X)=f'(X)$ for every $X\in \cX(D)$ since we have $$f(X_j)=f(X\cup X_j)\subseteq f(X)\;\textup{ and }\;f'(X_j)=f'(X\cup X_j)\subseteq  f'(X)$$ for $j$ large enough. We verify 
that $f(X_i)=f'(X_i)$ for every $i\in \N$ as follows: First note that the sequence $(f(X_i))_{i\in\N}$ is descending, because $$f(X_{i+1})= f(X_i\cup X_{i+1})\subseteq f(X_i).$$ Hence $f(X_i)$ contains all but finitely many vertices from $U$ for every $i\in \N$. In particular $f(X_i)$ is the only strong component of $D-X_i$ that contains infinitely many vertices from $U$. As a consequence we have $f(X_i)=f'(X_i)$ for every~$i\in~\N$.
\end{proof}

\begin{proof}[Proof of Theorem~\ref{thm: directions}]
It is straightforward to show that the map $\omega\mapsto f_\omega$ with domain $\Omega(D)$ and codomain the set of vertex-directions of $D$ is injective; we prove that it is onto. So given a vertex-direction $f$ of $D$ we need to find an end $\omega \in \Omega(D) $ such that $f_\omega = f$. 
Let $S_1^*$ be the set of all the vertices that dominate $f$ and $S_2^*$ the set of all the vertices that reverse dominate $f$. We split the proof into three cases:

First, assume that both $S_1^*\cap f(X)$ and $S_2^*\cap f(X)$ are non-empty for every $X\in \cX(D)$. Then $f$ is a vertex-direction in the closure of $\cU$ for $\cU:= \{S_1^*,S_2^* \} $. By Lemma \ref{lemma: connection closure; direction; necklace} we find an end $\omega$ in the closure of $\cU$ and we claim that $f_\omega = f$. Indeed, given $X\in \cX(D)$ we need to show that $C(X,\omega)=f(X)$. We may view the strong components of $D - X$ partially ordered by $C_1 \leq C_2$ if there is a path in $D-X$ from $C_1$ to $C_2$. By the order-extension-principle we choose a linear extension of $\leq$. Let $$(A_1,B_1):=(\bigcup \cA_1\cup X,X\cup \bigcup \cB_1),$$ where $\cA_1$ consists of all the strong components of $D-X$ strictly smaller than $f(X)$ and $\cB_1$ of all the others. Then $(A_1,B_1)$ points towards $f$. Since $S_2^*$ consists of the vertices reverse dominating $f$ we have $S_2^*\subseteq  B_1$. Since $\omega$ is in the closure of $\cU$ we have $C(X,\omega)\in \cB_1$.
Similarly, let $$(A_2,B_2):=(\bigcup \cA_2\cup X,X\cup \bigcup \cB _2),$$ where $\cA_2$ consists of all the strong components of $D-X$ smaller or equal to $f(X)$ and $\cB_2$ of all the others. Analogously to the argumentation for $C(X,\omega)\in \cB_1$, one finds out that $C(X,\omega)\in \cA_2$; together $C(X,\omega) \in \cB_1 \cap \cA_2$. Now, $f(X)=C(X,\omega)$ follows from the fact that $f(X)$ is the only element in the intersection $\cB_1\cap\cA_2$.

Second, suppose that $S_1^*\cap f(X)$ is empty for some $X\in \cX(D)$. If even $S_1^*$ is empty and $D$ strongly connected, then Lemma~\ref{lemma: sequence of separations}
and Lemma~\ref{lemma: connection closure; direction; necklace} do the rest: with $U$ as in the `moreover' part of Lemma~\ref{lemma: sequence of separations}, we have that $f$ is the unique vertex-direction in the closure of $U$ and Lemma~\ref{lemma: connection closure; direction; necklace} yields an end $\omega$ in the closure of $U$; by uniqueness $f_\omega=f$. 

In the following we will argue that we may assume $S_1^*$ to be empty and $D$ to be strongly connected. Fix $X'\in\cX(D)$ with $S_1^*\cap f(X^\prime)=\emptyset$. Let $D'=f(X')$ and let $f'$ be the vertex-direction of $D'$ induced by $f$, i.e., $f'$ sends a finite vertex set $X\subseteq V(D')$ to $f(X\cup X')$. Then the set of all the vertices that dominate $f'$ is empty: If $(A,B)$ is a finite order separation of $D$ that points away from $f$, then $(A\cap V(D'), B\cap V(D'))$ is a finite order separation of $D^\prime$ that points away from $f'$. As a consequence, any vertex from $D'$ that dominates $f'$ also dominates $f$, which means there is none.

Now, consider the end $\omega'$ of $D'$ with $f_{\omega'}=f'$ and the unique end $\omega$ of $D$  that contains $\omega'$ as a subset (of rays). We claim $f_{\omega}=f$, a fact that we verify as follows.
First observe that for $X\in \cX(D)$ with $X'\subseteq X$ we have 
$$f_\omega(X) 
=f_{\omega'}(X\cap V(D'))
=f'(X\cap V(D'))
= f(X).$$
Now let $X$ be an arbitrary finite vertex set of $D$. Since $f$ and $f_\omega$ are vertex-directions we have that $f(X\cup X')\subseteq f(X)$ and $f_\omega(X\cup X')\subseteq f_\omega (X)$. Furthermore, by our observation we have $f(X\cup X')= f_\omega(X\cup X')$. Hence also $f(X)=f_\omega (X)$ using that both $f(X)$ and $f_\omega(X)$ are strong components of $D-X$. 

Finally, the proof of the last case, that $S_2^*\cap f(X)$ is empty for some $X\in\cX(D)$, is analogue to the proof of the second case.
\end{proof}

\section{Limit edges and edge-directions}\label{section: limit edges}
\noindent In this section, we investigate limit edges of digraphs. Recall that, for two distinct ends $\omega, \eta\in \Omega(D)$, we call the pair $(\omega,\eta)$ a \emph{limit edge} from $\omega$ to $\eta$, if $D$ has an edge from $C(X,\omega)$ to $C(X,\eta)$ for every finite vertex set $X\subseteq V(D)$ that separates $\omega$ and $\eta$. 
For a vertex $v\in V(D)$ and an end $\omega\in \Omega(D)$ we call the pair $(v,\omega)$ a \emph{limit edge} \emph{from} $v$ \emph{to} $\omega$ if $D$ has an edge from $v$ to $C(X,\omega)$ for every finite vertex set $X\subseteq V(D)$ with $v \not\in  C(X, \omega)$. Similarly, we call the pair $(\omega,v)$ a \emph{limit edge} \emph{from} $\omega$ \emph{to} $v$ if $D$ has an edge from $C(X,\omega)$ to $v$ for every finite vertex set $X\subseteq V(D)$ with $v \not\in C(X, \omega)$. We write $\Lambda(D)$ for the set of limit edges of $D$. 
As we do for `ordinary' edges of a digraph, we will suppress the brackets and the comma in our notation of limit edges. For example we write $\omega\eta$ instead of $(\omega,\eta)$ for a limit edge between ends $\omega$ and $\eta$.

We begin this section with two propositions (Proposition~\ref{proposition: char limit edge type I} and Proposition~\ref{proposition: char limit edge type II}) saying that limit edges are witnessed by subdigraphs that are essentially the digraphs in Figure~\ref{fig:ladder} or Figure~\ref{fig: dominated ray}. Subsequently we prove Theorem~\ref{thm: edge-direction}.

Let $D$ be any digraph and let $\omega\in \Omega(D)$. With a slight abuse of notation, we say that a necklace $N\subseteq D$ \emph{represents} an end $\omega$ of $D$ if one (equivalently every) ray in $N$ represents~$\omega$. Note that for every end $\omega$ there is a necklace that represents $\omega$. Indeed,  apply the necklace lemma to any ray that represents $\omega$. 
\begin{proposition}\label{proposition: char limit edge type I}
For a digraph $D$ and two distinct ends $\omega$ and $\eta$ of $D$ the following assertions are equivalent:
\begin{enumerate}
    \item $D$ has a limit edge from $\omega$ to $\eta$;
    \item there are necklaces $N_\omega\subseteq D$ and $N_\eta\subseteq D$  that represent $\omega$ and $\eta$ respectively such that every bead of $N_\omega$ sends an edge to a bead of $N_\eta$. 
\end{enumerate}
Moreover, the necklaces may be chosen disjoint from each other and such that the $n$th bead of $N_\omega$ sends an edge to the $n$th bead of $N_\eta$.
\end{proposition}

\begin{proof}
We begin with the forward implication (i)$\to$(ii). By possibly deleting a finite vertex set of $D$ that separate $\omega$ and $\eta$, we may assume that $\omega$ and $\eta$ live in distinct strong components of $D$. Given a necklace $N$ let us write $N[n,m]$ for the inflated symmetric path from the $n$th bead to the $m$th bead of $N$ and $N[n]$ for the inflated symmetric path from the first bead to the $n$th bead of $N$.  First, let us  fix auxiliary necklaces $N_\omega^\prime\subseteq D$  and $N_\eta^\prime\subseteq D$ that represent $\omega$ and $\eta$, respectively.

We inductively construct sequences $(N_\alpha^n)_{n \in \N}$ of necklaces, for $\alpha \in \{ \omega , \eta \}$, so that $N_\alpha^{n}[n-1] = N_\alpha^{n-1}[n-1]$  and the $n$th bead of $N_\omega^n$ sends an edge to the $n$th bead of $N_\eta^n$. Furthermore, we will make sure that $N^\prime_\alpha[n]\subseteq N_\alpha^n[n]$.

Then the unions $\bigcup \{\,N_\alpha^{n}[n]\mid n\in \N\,\}$ define  necklaces $N_\alpha$, for $\alpha \in \{ \omega ,\eta \}$,  as desired. Indeed, as $N_\alpha$ includes $N^\prime_\alpha$ it  also represents $\alpha$. Note, that our construction yields the `moreover' part.
Let $n\in \N$ and suppose that $N_\omega^n$ and $N_\eta^n$ have already been constructed. Let $X$ be the union of $N^n_\omega[n]$, $N_\eta^n[n]$ and the two paths between the $n$th bead and the $(n+1)$th bead of $N^n_\omega$ and $N^n_\eta$, respectively. So $X$ might be empty for $n=0$. Note that by our assumption $\omega$ and $\eta$ live in distinct strong components of $D$, so in particular they also live in distinct strong components of $D-X$. As $D$ has a limit edge from $\omega$ to $\eta$ we find an edge $e$ from $C(X,\omega)$ to $C(X,\eta)$. Fix a finite strongly connected vertex set $Y_\alpha\subseteq C(X,\alpha)$ that includes $N^n_\alpha[n+1,m]$ for a suitable $m\ge n+1$ and the endvertex of $e$ in $C(X,\alpha)$ but that avoids the rest of $N^n_\alpha$ for $\alpha \in\{\omega,\eta\}$. Replacing the inflated symmetric subpath $N^n_\alpha[n+1,m]$ by $Y_\alpha$ and declaring $Y_\alpha$ as the $(n+1)$th bead of $N_{\alpha}^{n+1}$ for $\alpha\in\{\omega, \eta\}$ yields necklaces $N_\omega^{n+1}$ and $N_\eta^{n+1}$ that are as desired.

Now, let us prove the backward implication (ii)$\to$(i). As every finite vertex set $X$ meets only finitely many beads of $N_\omega$ and $N_\eta$ there are beads of $N_\omega$ and $N_\eta$ that are included in $C(X,\omega)$ and $C(X,\eta)$, respectively. Hence, if $X$ separates $\omega$ and $\eta$, there is an edge from $C(X,\omega)$ to $C(X,\eta)$.
\end{proof}
\noindent There is a natural partial order on the set of ends, where $\omega \leq \eta$ if for every two rays $R_\omega$ and $R_\eta$ that represent $\omega$ and $\eta$, respectively, there are infinitely many pairwise disjoint paths from $R_\omega$ to $R_\eta$. By Proposition~\ref{proposition: char limit edge type I} we have that $\omega\le \eta$, whenever $\omega\eta$ is a limit edge for ends $\omega$ and $\eta$. The converse of this is in general false, for example in the digraph that is obtained from the digraph in Figure~\ref{fig:ladder} by subdividing every vertical edge once. 

\begin{proposition}\label{proposition: char limit edge type II}
For a digraph $D$, a vertex $v$ and an end $\omega$ of $D$ the following assertions are equivalent:
\begin{enumerate}
    \item{$D$ has a limit edge from $v$ to $\omega$ (from $\omega$ to $v$);}
    \item{there is a necklace $N\subseteq D$ that represents $\omega$ such that $v$ sends (receives) an edge to (from) every bead of $N$.}
\end{enumerate}
\end{proposition}
\begin{proof} 
We consider the case that $v$ sends an edge to every bead of $N$; for the other case consider the reverse of $D$.

For the forward implication (i)$\rightarrow$(ii) a similar recursive construction as in the proof of Proposition~\ref{proposition: char limit edge type I} yields a necklace $N$ as desired.

Now, let us prove the backward implication (ii)$\to$(i). As every finite vertex set $X$ hits only finitely many beads of $N$, there is one bead that is contained in $C(X, \omega)$. Therefore  there is an edge from $v$ to $C(X, \omega)$  whenever $v\notin C(X,\omega)$. 
\end{proof}
\noindent As a consequence of this proposition, every vertex $v\in V(D)$ for which $D$ has a limit edge from $v$ to an end $\omega\in \Omega(D)$ dominates $\omega$. The converse of this is in general false, for example in the digraph that is obtained from the digraph in Figure~\ref{fig: dominated ray}  by subdividing every edge once. Similarly, if $\omega v$ is a limit edge between an end $\omega$ and a vertex $v$, then $v$ reverse dominates $\omega$; the converse is again false in general.

Now, let us turn to our second type of directions. We call a direction $f$ of $D$ an \emph{edge-direction}, if there is some $X \in \cX(D)$ such that $f(X)$ is a bundle of $D-X$, i.e., if $f$ is not a vertex-direction. 
Recall that every end defines a vertex-direction. Similarly, every limit edge $\lambda$ defines an edge-direction as follows. 

We say that a limit edge $\lambda=\omega\eta$ \emph{lives} in the bundle defined by $E(X, \lambda)$ if $X \in \cX(D)$ separates $\omega$ and $\eta$.  If $X \in \cX(D)$ does not separate $\omega$ and $\eta$, we say that $\lambda=\omega\eta$ \emph{lives} in the strong component $C(X, \omega)= C(X, \eta)$ of $D-X$.
We use similar notations for limit edges of the form $\lambda= v \omega$ or $\lambda=  \omega v$ with $v\in V(D)$ and $\omega\in \Omega(D)$: We say that a limit edge $\lambda$  \emph{lives in} the bundle $E(X,\lambda )$ if $v \not\in C(X,\omega)$ and we say that $\lambda$  \emph{lives in} the strong component $C(X, \omega)$ of $D-X$, if $v \in C(X, \omega)$.

The edge-direction $f_\lambda$ defined by $\lambda$ is the edge-direction that sends every finite vertex set $X\subseteq V(D)$ to the bundle or strong component of $D-X$ in which $\lambda$ lives. Our next theorem states that there is a one-to-one correspondence between the edge-directions of a digraph and its limit edges:

\begin{mainresult}\label{thm: edge-direction}
Let $D$ be any infinite digraph. The map $\lambda \mapsto f_{\lambda}$ with domain $\Lambda(D)$ is a bijection between the  limit edges and the  edge-directions of $D$.
\end{mainresult}
\begin{proof}
It is straightforward to show that the map given in (ii) is injective; we prove onto. 
So let $f$ be any edge-direction of $D$. First suppose that $f(X)$ is always a strong component or a bundle between strong components for every $X\in \cX(D)$.  Then $f$ defines two vertex-directions $f_1$ and $f_2$ as follows. If $f(X)=E(C_1,C_2)$ is a bundle then let $f_1(X)=C_1$ and $f_2(X)=C_2$. Otherwise, $f(X)$ is a strong component and we put $f_1(X)=f_2(X)=f(X)$. Now, the inverse of the function from Theorem~\ref{thm: directions} returns ends $\omega$ and $\eta$ for $f_1$ and $f_2$, respectively. We conclude that~$\omega\eta$ is a limit edge and that $f=f_{\omega\eta}$.  

Now, suppose that $f$ maps some finite vertex set $X'$ to a bundle between a vertex $v \in X'$ and a strong component of $D-X'$. Then also $f(\{v\})$ is a bundle between $v$ and a strong component. We consider the case where $f(\{v\})$ is of the form $E(v,C_v)$ for some strong component $C_v$ of $D-v$; the other case  is analogue.

Let us define a vertex-direction $f^\prime$ of $D$. First, for every $X \in \cX(D)$ with $v \in X$ we have that $f(X)$ is a bundle of the form $E(v,C) $ for a strong component $C$ of $D-X$ and we put $f^\prime(X) =C$. Second, if $v \notin X$ for some $X \in \cX(D)$ we have that  $f(X)$ is either a strong component $C^\prime$ of $D - X$ or a bundle $E(C,C^\prime)$ with $v \in C$. We then put $f^\prime(X)= C^\prime$. It is straightforward to check that $f^\prime$ is indeed a vertex-direction.
Finally, the inverse of the map from Theorem~\ref{thm: directions} applied to $f'$ returns an end $\omega$. By the definition of $f'$ we have that $v\omega$ is a limit edge of $D$ and a close look to the definitions involved points out that~$f=f_{v\omega}$. 
\end{proof}

\bibliographystyle{amsplain}
\bibliography{bibliography.bib}

\providecommand{\bysame}{\leavevmode\hbox to3em{\hrulefill}\thinspace}
\providecommand{\MR}{\relax\ifhmode\unskip\space\fi MR }
\providecommand{\MRhref}[2]{%
  \href{http://www.ams.org/mathscinet-getitem?mr=#1}{#2}
}
\providecommand{\href}[2]{#2}
\begin{thebibliography}{10}

\bibitem{bang2008digraphs}
J.~Bang-Jensen and G.~Gutin, \emph{Digraphs: theory, algorithms and
  applications}, Springer Science \& Business Media, 2008.

\bibitem{duality}
H.~Bruhn and R.~Diestel, \emph{Duality in infinite graphs}, Comb.,\ Probab. \&
  Comput. \textbf{15} (2006), 75--90.

\bibitem{UnfriendlyPartition}
H.~Bruhn, R.~Diestel, A.~Georgakopoulos, and P.~Spr{\"u}ssel, \emph{Every
  rayless graph has an unfriendly partition}, Combinatorica \textbf{30} (2010),
  no.~5, 521--532.

\bibitem{StarComb1StarsAndCombs}
C.~Bürger and J.~Kurkofka, \emph{{Duality theorems for stars and combs I:
  Arbitrary stars and combs}}, 2020,
  \href{https://arxiv.org/abs/2004.00594}{available at arXiv:2004.00594}.

\bibitem{EndsOfDigraphsII}
C.~Bürger and R.~Melcher, \emph{{Ends of digraphs II: The topological point of
  view}}, 2020,
  \href{https://www.math.uni-hamburg.de/home/buerger/Ends_of_digraphs2.html}{available
  at arXiv:2004.00591}.

\bibitem{Caratheodory1913}
C.~Carath{\'e}odory, \emph{{{\"U}ber die Begrenzung einfach
  zusammenh{\"a}ngender Gebiete}}, Math. Annalen \textbf{73} (1913), no.~3,
  323--370.

\bibitem{DiestelBook3}
R.~Diestel, \emph{{Graph Theory}}, 3th ed., Springer, 2005.

\bibitem{diestel2006end}
\bysame, \emph{End spaces and spanning trees}, J. Combin.\ Theory (Series B)
  \textbf{96} (2006), no.~6, 846--854.

\bibitem{DiestelBook5}
\bysame, \emph{{Graph Theory}}, 5th ed., Springer, 2016.

\bibitem{EndsAndTangles}
\bysame, \emph{{Ends and Tangles}}, {Abh.\ Math.\ Sem.\ Univ.\ Hamburg}
  \textbf{87} (2017), no.~2, 223--244, {Special issue in memory of Rudolf
  Halin}, \href{https://arxiv.org/abs/1510.04050}{arXiv:1510.04050v3}.

\bibitem{CyclesI}
R.~Diestel and D.~K\"uhn, \emph{{On Infinite Cycles {I}}}, Combinatorica
  \textbf{24} (2004), 68--89.

\bibitem{Diestel_Kuehn_Directions_03}
\bysame, \emph{{Graph-theoretical versus topological ends of graphs}}, J.
  Combin.\ Theory (Series B) \textbf{87}
  (\href{https://www.math.uni-hamburg.de/home/diestel/papers/TopEnds.pdf}{2003}),
  197--206.

\bibitem{Freudenthal1931Enden}
H.~Freudenthal, \emph{{{\"U}ber die Enden topologischer R{\"a}ume und
  Gruppen}}, Mathematische Zeitschrift \textbf{33} (1931), no.~1, 692--713.

\bibitem{AgelosFleischner}
A.~Georgakopoulos, \emph{Infinite hamilton cycles in squares of locally finite
  graphs}, Advances in Mathematics \textbf{220} (2009), no.~3, 670--705.

\bibitem{Halin_Enden64}
R.~Halin, \emph{{Über unendliche Wege in Graphen}}, Math.\ Annalen
  \textbf{157} (1964), 125--137.

\bibitem{jung69}
H.A. Jung, \emph{{Wurzelb{\"a}ume und unendliche Wege in Graphen}}, Math.\
  Nachr. \textbf{41} (1969), 1--22.

\bibitem{ApproximatingNormalTrees}
J.~Kurkofka, R.~Melcher, and M.~Pitz, \emph{Approximating infinite graphs by
  normal trees}, 2020,
  \href{https://arxiv.org/abs/2002.08340}{arXiv:2002.08340}.

\bibitem{EndsTanglesCrit}
J.~Kurkofka and M.~Pitz, \emph{Ends, tangles and critical vertex sets}, Math.\
  Nachr. \textbf{292} (2019), no.~9, 2072--2091,
  \href{https://arxiv.org/abs/1804.00588}{arXiv:1804.00588}.

\bibitem{StoneCechTangles}
\bysame, \emph{{Tangles and the Stone-\v{C}ech compactification of infinite
  graphs}}, J.~Combin.\ Theory (Series B) \textbf{146} (2021), 34--60,
  \href{https://arxiv.org/abs/1806.00220}{arXiv:1806.00220}.

\bibitem{neumanndichromatic}
V.~Neumann-Lara, \emph{The dichromatic number of a digraph}, Journal of
  Combinatorial Theory, Series B \textbf{33} (1982), no.~3, 265--270.

\bibitem{polat1996ends}
N.~Polat, \emph{Ends and multi-endings, {I}}, J. Combin.\ Theory (Series B)
  \textbf{67} (1996), 86--110.

\bibitem{polat1996ends2}
\bysame, \emph{Ends and multi-endings, {II}}, J. Combin.\ Theory (Series B)
  \textbf{68} (1996), 56--86.

\bibitem{Schmidt1983}
R.~Schmidt, \emph{{Ein Ordnungsbegriff f{\"u}r Graphen ohne unendliche Wege mit
  einer Anwendung auf $n$-fach zusammenh{\"a}ngende Graphen}}, Arch.\ Math.
  \textbf{40} (1983), no.~1, 283--288.

\bibitem{Zuther96Thesis}
J.~Zuther, \emph{Planar strips and an end concept for digraphs}, Ph.D. thesis,
  Technische Universit{\"a}t Berlin, 1996.

\end{thebibliography}
\end{document}